\documentclass[11pt]{article}

\usepackage[margin=7em]{geometry}

\usepackage{amsfonts}
\usepackage{amsmath}
\usepackage{amsthm}
\usepackage{amssymb}

\usepackage{graphics}
\usepackage{epsfig}
\usepackage{enumitem}

\makeatletter
\def\blfootnote{\gdef\@thefnmark{}\@footnotetext}
\makeatother

\theoremstyle{plain}
    \newtheorem{theorem}{Theorem}
    \newtheorem{lemma}{Lemma}
    \newtheorem{proposition}{Proposition}
    \newtheorem{corollary}{Corollary}

\theoremstyle{definition} 
    \newtheorem{definition}{Definition}
    \newtheorem{remark}{Remark}

\theoremstyle{remark} 

\newcommand\mnote[1]{} 
\newcommand\be{\begin{equation}}

\newcommand\ee{\end{equation}}

\newcommand{\comment}[1]{}

\newcommand{\N}{{\mathbb N}}
\newcommand{\Np}{\mathbb{N}^+}

\newcommand{\R}{{\mathbb R}}
\newcommand{\CC}{{\mathbb C}}

\newcommand{\ev}{\mbox{\bf E}}

\newcommand{\one}{{\mathbf 1}}

\newcommand{\sm}{{\raise0.3ex\hbox{$\scriptstyle \setminus$}}}

\newcommand{\nid}{\noindent}
\newcommand{\gG}{\Gamma}
\newcommand{\gga}{\gamma}
\newcommand{\gd}{\delta}
\newcommand{\gep}{\varepsilon}
\newcommand{\gt}{\theta}
\newcommand{\gl}{\lambda}
\newcommand{\gL}{\Lambda}

\newcommand{\tr}{\operatorname{tr}}

\newcommand{\B}[1]{\textbf{#1}}

\newcommand{\C}[1]{\mathcal{#1}}
\newcommand{\D}[1]{\mathbb{#1}}

\newcommand{\ol}[1]{\overline{#1}}

\renewcommand{\qedsymbol}{\ensuremath{\blacksquare}}
\renewcommand\Re{\operatorname{Re}}
\renewcommand\Im{\operatorname{Im}}

\usepackage{natbib}

\begin{document}

\title{Triangular random matrices and biorthogonal ensembles}

\author{
\renewcommand{\thefootnote}{\arabic{footnote}}
Dimitris Cheliotis
\footnotemark[1]}

\footnotetext[1]{
Department of Mathematics,  University of Athens,  Panepistimiopolis 15784,  Athens 
Greece.
}

\blfootnote{
 MSC2010 Subject Classifications: 60B20, 60F15, 62E15}
 
\blfootnote{Keywords: Triangular random matrices, singular eigenvalues, alternating trees, determinantal processes, biorthogonal ensembles, DT-operators, Lambert function.}

\date{April 18, 2014}

\maketitle

\begin{abstract}
We study the singular values of certain triangular random matrices. When their elements are i.i.d. standard complex Gaussian random variables, the squares of the singular values form a biorthogonal ensemble, and with an appropriate change in the distribution of the diagonal elements, they give the biorthogonal Laguerre ensemble. For triangular Wigner matrices, we give alternative proofs for the convergence of the empirical distribution of the appropriately scaled squares of the singular eigenvalues to a distribution with support $[0, e]$, as well as for the almost sure convergence of the rescaled largest singular eigenvalue to $\sqrt{e}$ under the additional assumption of mean zero and finite fourth moment for the law of the matrix elements.     
\end{abstract}

\section{Introduction and statement of the results}

\subsection{Singular values of random matrices}

Singular values of random matrices are of importance in numerical analysis, multivariate statistics, information theory, and the spectral theory of random  non-symmetric matrices. See the survey paper \cite{Cha09}.

We state in this subsection three of the very basic results concerning singular values of random matrices that are relevant to our work.

Let $\{X_{i, j}: i, j\in \Np\}$ be i.i.d. complex valued random variables with variance 1, and for $n, m\in \Np$ consider the $n\times m$ matrix $X(n, m):=(X_{i, j})_{1\le i \le n, 1\le j \le m}$. Call $\gl_1^{n, m}\ge \gl_2^{n, m}\ge \dots \ge \gl_n^{n, m}\ge 0$ the eigenvalues of the Hermitian, positive definite matrix
$$S_{n, m}=\frac{1}{m} X(n, m) X(n, m)^*,$$
and 
$$L_{n, m}:=\frac{1}{n}\sum_{i=1}^n \gd_{\gl_i^{n, m}}$$
their empirical distribution. It was shown in \cite{MP} that
for $c>0$,  with probability  1, as $n, m\to\infty$ so that $n/m\to c$,  $L_{n, m}$ converges weakly to the measure 
\begin{equation} \label{MPDensity}
\one_{a\le x\le b} \frac{1}{2\pi x c} \sqrt{(b-x)(x-a)} \, dx+\one_{c>1}\left(1-\frac{1}{c}\right)\gd_0
\end{equation}   
where $a=(1-\sqrt{c})^2, \, b=(1+\sqrt{c})^2$.

Regarding the largest eigenvalue, it was  proved in \cite{GE} under certain moment assumptions, that with probability 1, $\gl_1^{n, m}$ converges to $b$ as $n, m\to\infty$. Then \cite{BY} showed that this convergence takes place under the assumption that $|X_{1, 1}|$ has finite fourth moment and that this assumption is necessary for the validity of the conclusion.

When the $X_{i, j}$ follow the standard complex Gaussian distribution and $n\le m$, the vector $(\gl_1^{n, m},  \gl_2^{n, m}, \ldots, \gl_n^{n, m})$ has density with respect to Lebesgue measure in $\R^n$ which is 
\begin{equation} \label{CWishart} 
\frac{1}{\prod_{k=1}^n \Gamma(m-n+k) \Gamma(k)}\, e^{-\sum_{k=1}^n x_k} \Big(\prod_{k=1}^n x_i\Big)^{m-n} \prod_{1\le i \le j \le n} (x_i-x_j)^2 \one_{x_1>x_2>\cdots>x_n>0}.
\end{equation}
See, for example, relation (3.16) in \cite{For}.

\subsection{Triangular Wigner matrices} 

In this work, we study the singular values of certain triangular random matrices. The motivation comes from the purely mathematical viewpoint as triangular matrices are ingredients in several matrix decompositions.  The results of this subsection have appeared before, and we offer alternative proofs. 

Assume as above that $\{X_{i, j}: i, j \in \Np, i\ge j,\}$ are i.i.d. complex
valued with variance 1, and for $n\in\Np$ let $X(n)$ be the lower triangular
$n\times n$ matrix  whose $(i, j)$ element is $X_{i, j}$ for
$1\le j\le i\le n$. Call $ \gl_1^{(n)}\ge \gl_2^{(n)}\ge \dots \ge \gl_n^{(n)}\ge 0$ the eigenvalues of the Hermitian matrix
$$S_n=\frac{1}{n} X(n) X(n)^*,$$
and 
$$L_n:=\frac{1}{n}\sum_{i=1}^n \gd_{\gl_i^{(n)}}$$
their empirical distribution.   

The fact that $L_n$ converges weakly and description of the limit was given in \cite{DH}.

\begin{theorem} \label{convergence} With probability 1,
$(L_n)_{n\ge1}$ converges weakly to a deterministic measure $\mu_0$ on $\R$
with moments 
\begin{equation}
\int_\R x^k d\mu_0(x)=\frac{k^k}{(k+1)!}
\end{equation}
for all $k\in \N$.
\end{theorem}

The measure $\mu_0$ comes from a density which can be expressed in terms of the Lambert function. This is a multivalued function, it is the inverse of $w\mapsto we^w$. We will use the the principal branch, $W$, of this inverse, which is defined in $\CC\sm (-\infty, -e^{-1}]$.
$W$ is analytic in $\CC\sm(-\infty, -e^{-1}]$ and can be
extended to $\CC$ so that it is continuous on the closed upper half plane (See Section 4 of \cite{CGHJK}). Below, $W$ will denote this extention.

\begin{corollary} \label{AnalyticProp} The measure $\mu_0$ has

\begin{enumerate}

\item continuous density $f_0$ with support $[0, e]$.

\item Stieltjes transform
\begin{equation}
S(z)=-1-\frac{1}{z W(-z^{-1})}=-1+e^{W(-z^{-1})}
\end{equation} 
for all $z\in \CC$ with $\Im(z)>0$.

\item $R$-transform
$$R(z)=-\frac{1}{(1-z) \log(1-z)}-\frac{1}{z}$$
for all $z\in\CC$ with $|z|<1$. 
\end{enumerate}

\end{corollary}

\begin{figure}[t] 
  \centering
  \includegraphics[width=25em]{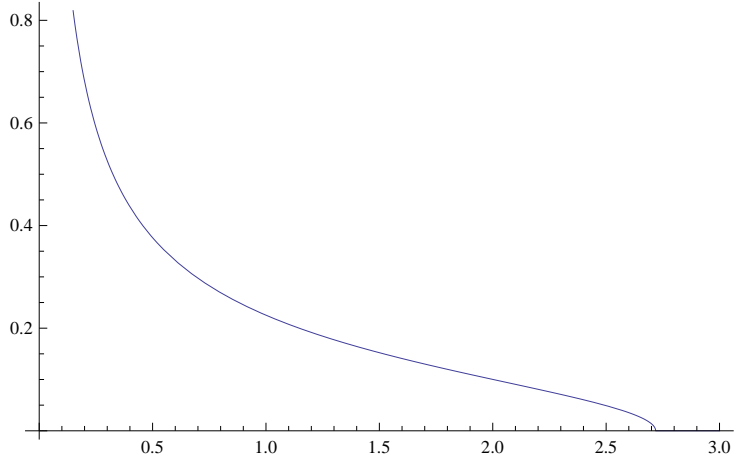}
  \caption{The graph of the density
$f_0(x)=\pi^{-1}\Im\left\{-\frac{1}{xW(-x^{-1})}\right\}$. Its support
is $[0,
e]$.}\label{density}
\end{figure}

\nid The graph of the density $f_0$ is shown in Figure
\ref{density}.

\begin{remark} \label{edgesRem}
Near 0, the density $f_0$ behaves as 
$$f_0(x)\sim \frac{1}{x(W(x^{-1}))^2} \sim \frac{1}{x(\log
x)^2},$$
so that it explodes much faster than $1/\sqrt{x}$, which is the speed of explosion of the Marchenko-Pastur density in the case $m=n$.   Near $e$, $f_0$ behaves as 
\begin{equation}
f_0(x)\sim \frac{\sqrt{2}}{\pi e^{3/2}} \sqrt{e-x}.
\end{equation}
We prove these statements in Subsection \ref{densityPropertiesSec}.
\end{remark}

The next result, which is analogous to the result of Yin and Bai, appeared in the recent preprint \cite{CGLZ}. 

\begin{theorem} \label{MaxEVConv}
Assume that $X_{1, 1}$ has mean 0, variance 1, and finite fourth moment. Then with probabilty 1, $\lim_{n\to\infty} \gl_1^{(n)}=e$.
\end{theorem}

\begin{remark} 
When this research begun, a few months ago, we were not aware that the result of Theorem \ref{convergence} was known. 
In \cite{DH}, the theorem is phrased in the language of free probability, and its proof uses tools from that area together with combinatorial arguments.

The proof of Theorem \ref{MaxEVConv} in \cite{CGLZ} uses probabilistic and operator theoretic arguments.

Our proofs of Theorems \ref{convergence} and \ref{MaxEVConv}  follow the classical method of moments and path counting used for the analogous theorems for Wigner and sample covariance matrices (see e.g., Chapter 2 in \cite{Tao}). The crucial ingredient in our analysis is the notion of rooted alternating plane tree, which appears because of the triangular structure of the matrix.   
\end{remark}

\subsection{The complex Gaussian case}

In the case that the random variables $\{X_{i, j}: i, j\in \Np, \, i\ge j\}$ in the previous subsection are complex standard normal, we can compute explicitly the joint distribution of the eigenvalues of $X(n) X(n)^*$.

\begin{theorem} \label{JointEvalues} For each positive integer $n$:
\begin{enumerate}
\item
The vector $\gL_n:=(\gl_1^{(n)},
\gl_2^{(n)}, \ldots, \gl_n^{(n)})$ of the eigenvalues $\gl_1^{(n)}\ge
\gl_2^{(n)}\ge\cdots\ge \gl_n^{(n)}$ of $X(n)X^*(n)$ has density  given by
\begin{equation} \label{JointEvaluesEq}
f_{\gL_n}(x_1, x_2, \ldots, x_n)=\frac{1}{\prod_{j=1}^{n-1}
j!}\,  e^{-\sum_{j=1}^nx_j} \prod_{i<j}(x_i-x_j)(\log x_i-\log
x_j) \one_{x_1> x_2> \cdots>x_n>0}
\end{equation}

\item The point process $\{\gl_1^{(n)}, \gl_2^{(n)}, \ldots, \gl_n^{(n)}\}$ is determinantal.
\end{enumerate}
\end{theorem}
\nid The theorem will be implied by the more general Theorems \ref {JointEvaluesGen} and \ref{determinantalThm} of the next subsection.

\subsection{Eigenvalue realization of the biorthogonal Laguerre ensemble} \label{LaguerreSec}

Fix a positive integer $n$. Consider the lower triangular matrix $(X_{i, j})_{1\le i, j\le n}$ with $\{X_{i, j}: 1\le j <i\le n\}$ standard complex normal variables and $X_{k, k}$ having density 
\begin{equation} \label{XkkDensity}
f_k(z)=\frac{1}{\pi \Gamma(c_k)} \, e^{-|z|^2} |z|^{2(c_k-1)}
\end{equation}
for all $z\in\CC$, where for $(c_k)_{1\le k\le n}$ we assume that they form an arithmetic progression with difference $\theta\in\D{R}$ so that all its terms are positive. Thus $X_{k, k}$ can be written as 
\begin{equation} \label{DiagonalLaws}
X_{k, k}=\frac{1}{\sqrt{2}} \, e^{i \phi_k} Y_k
\end{equation}
where $Y_k$ follows the $\chi_{2 c_k}$ distribution and $\phi_k$ is uniform on $[0, 2\pi)$ independent of $Y_k$.

It is enough to consider the case that $\theta\ge0$, because otherwise taking $P$ the matrix with ones in the antidiagonal, we see that the matrix $(P X P^*)^*$ becomes of the form we consider with difference $\tilde \theta=-\theta$.
Thus, we assume that there are $\theta\ge0,\, b>0$ so that 
\begin{equation}
c_k=\theta(k-1)+b
\end{equation} 
for all $k\in\{1, 2, \ldots, n\}$. We denote the matrix constructed with this prescription by $X^{\theta, b}(n)$. For the  distribution of the squares of its singular eigenvalues we have the following theorem.

\begin{theorem} \label{JointEvaluesGen}
The vector $\gL_n:=(\gl_1^{(n)},
\gl_2^{(n)}, \ldots, \gl_n^{(n)})$ of the eigenvalues $\gl_1^{(n)}\ge
\gl_2^{(n)}\ge\cdots\ge \gl_n^{(n)}\ge0$ of $X^{\theta, b}(n)X^{\theta, b}(n)^*$ has density $f_{\gL_n}(x_1, x_2, \ldots, x_n)$ given by
\begin{align} \label{JointEvaluesEqGenPos}
&\frac{1}{\prod_{j=1}^{n-1}
j!}\, \frac{\theta^{-n(n-1)/2}}{\prod_{k=1}^n \Gamma(c_k)}\,  e^{-\sum_{j=1}^nx_j} \Big(\prod_{j=1}^n x_j^{b-1}\Big)\prod_{1\le i<j\le n}(x_i-x_j)(x_i^\theta-x_j^\theta) \, \one_{x_1>x_2> \cdots>x_n>0}
\intertext{when $\theta>0$, and}
&\frac{1}{\prod_{j=1}^{n-1}
j!}\, \frac{1}{\Gamma(b)^n}\,  e^{-\sum_{j=1}^nx_j} \Big(\prod_{j=1}^n x_j^{b-1}\Big)\prod_{1\le i<j\le n}(x_i-x_j)(\log x_i-\log x_j) \, \one_{x_1>x_2>\cdots>x_n>0} \label{JointEvaluesEqGenZero}
\end{align}
when $\theta=0$.
\end{theorem}

\begin{remark} 
i) When $\theta=0$ and $b=1$, the matrix $X^{\theta, b}(n)$ is exactly $X(n)$ of the previous subsection. And thus we get Part (i) of Theorem \ref{JointEvalues}.

ii) When $\theta=1$ and $b=m-n$, with $m\ge n$ positive integers, \eqref{JointEvaluesEqGenPos} is the density \eqref{CWishart}. This is expected because there is a unitary matrix $U$ so that $X(m, n)U\overset{d}{=}[X^{1, m-n}(n), 0]$, where $0$ is the $n\times(m-n)$ zero matrix. 

iii) The density in \eqref{JointEvaluesEqGenPos} is the density of the $n$-point  biorthogonal Laguerre ensemble, so termed and studied in Section 4 of \cite{Bor}, with parameter pair $(\alpha, \theta)$ being $(b-1, \theta)$. Note that \eqref{JointEvaluesEqGenZero} is the $\theta\to0$ limit of \eqref{JointEvaluesEqGenPos}.

iv) Densities of the form \eqref{JointEvaluesEqGenPos} were introduced by \cite{Mut} in the context of disordered conductors. A good approximation for the conductance of such a conductor is given by the sum $\sum_{k=1}^n(1+x_k)^{-1}$, where $\{x_k:1\le k\le n\}$ are the eigenvalues of a certain positive definite random matrix.  It is asserted in the above reference that the assumption that these eigenvalues come from a density of the form \eqref{JointEvaluesEqGenPos}, with $\theta\ne1$, which has two two-body interaction terms, namely $\prod_{i<j}(x_i-x_j)$ and $\prod_{i<j}(x_i^\theta-x_j^\theta)$, matches better experimental measurements from the conductor in the metalic regime than models with only one such term, as is \eqref{CWishart}. It is also shown that this density defines a  determinantal point process. Later, \cite{Bor} gave an explicit formula for the kernel of the process, and using it determined the $n\to\infty$ limit at the hard edge (i.e., at 0) of an appropriate scaling of the process.  
\end{remark}

The formula for $f_{\gL_n}$ implies that $\{\gl_1^{(n)}, \gl_2^{(n)}, \ldots,
\gl_n^{(n)}\}$ is a biorthogonal ensemble (\cite{Bor}, \cite{For} Section
5.8). And this allows to prove with little effort  
that the ensemble is a determinantal point process. In the case $\theta>0$, this is already known. We cover next the $\theta=0$ case. Define
$$g_{j, k}:=\int_0^\infty x^j (\log x)^k e^{-x} \, dx$$
for $j, k\in\N$, and consider the matrix $G:=(g_{i, j})_{i, j\in\N}$. 
\begin{theorem} \label{determinantalThm}
For each positive integer $n$:
\begin{enumerate}
\item  The matrix $G^{(n)}:=(g_{j, k})_{0\le j, k\le n-1}$ is invertible. 

\item The point process $\{\gl_1^{(n)}, \gl_2^{(n)}, \ldots, \gl_n^{(n)}\}$ with law given by \eqref{JointEvaluesEqGenZero} is determinantal with kernel 
$$K_n(x, y)=e^{-\frac{x+y}{2}} (x y)^{\frac{b-1}{2}}\sum_{j, k=1}^n c_{j-1, k-1}^{(n)} (\log y)^{j-1}
x^{k-1}.$$
where $(c_{j, k}^{(n)})_{0\le j, k \le n-1}$ is the inverse of $G^{(n)}$.
\end{enumerate}
\end{theorem}

Finally, we come to the empirical spectral distribution $L_n^{\theta, b}$ of $X^{\theta, b}(n) X^{\theta, b}(n)^*/n$. The work in \cite{DH} implies that this  converges to a non trivial limit. To explain this connection, we need the notion of a $DT$-element.

Let $\nu$ a probability measure on $\CC$ with compact support, and $c>0$. For each $n$, let $T_n$ be an $n\times n$ matrix with $(T_n)_{i, j}=0$ if $1\le i\le j\le n$, and $\{(T_n)_{i, j}: 1\le j < i\le n\}$ i.i.d. standard complex Gaussian. And let $D_n$ be a diagonal $n\times n$ matrix with i.i.d. diagonal elements, each having law $\nu$, and independent of $T_n$. Let $Z_n:=D_n+c n^{-1/2} T_n$. It can be proved that for each $k\ge1$ and $\gep(1), \gep(2), \ldots, \gep(k)\in\{1, *\}$ the limit
\begin{equation} \label{DTLimits} \lim_{n\to\infty}\frac{1}{n}\ev(\tr\{Z_n^{\gep(1)} Z_n^{\gep(2)}\cdots Z_n^{\gep(k)}\})
\end{equation}
exists (Theorem 2.1 in \cite{DH}).

\begin{definition} An element $x$ of a $*$-noncommutative probability space $(\C{A}, \phi)$ is called a $DT(\nu, c)$-element if for every $k\ge1$ and $\gep(1), \gep(2), \ldots, \gep(k)\in\{1, *\}$, we have that 
$\phi(x^{\gep(1)} x^{\gep(2)}\cdots x^{\gep(k)})$ equals the value in \eqref{DTLimits}.
\end{definition}

And we are now ready to discuss the convergence of the sequence $(L_n^{\theta, b})_{n\ge1}$.

\begin{theorem} \label{DTLimitThm} The empirical distribution of the eigenvalues of $X^{\theta, b}(n) X^{\theta, b}(n)^*/n$ converges to a measure $\mu_\theta$ whose moments are the moments of $x x^*$ where $x$ is a $DT(\nu_\theta, 1)$ element, and $\nu_\theta$ is the uniform measure on the disc $D(0, \sqrt{\theta}):=\{z\in \D{C}: |z|\le \sqrt{\theta}\}$.  
\end{theorem}

Note that the limit does depend on $b$. In the case that $\theta>1$ and $b=1$, it is proven in \cite{CR}, see Paragraph 4.5.1, that the measure $\mu_\theta$ has density $f_\theta$ with support $I_\theta=[0, (1+\theta)^{1+1/\theta}]$. To describe it, let $J:\CC\sm [-1, 0]\to\CC$ with
$$J(z)=\theta(z+1)\left(\frac{z+1}{z}\right)^{1/\theta}.$$
For each $x$ interior point of $I_\theta$, there are exactly two solutions of $J(z)=x$, which are conjugate complex numbers. Call them $I_-(x), I_+(x)$ so that $\Im(I_+(x))>0$. Then the density $f_\theta$ is given by

\begin{equation}
f_\theta(x)=\begin{cases}\frac{\theta}{2\pi  x i}(I_+(x)-I_-(x)) & \text{ if } x\in \big(0, (1+\theta)^{1+1/\theta}\big), \\
0 & \text{ if } x\in \R\sm \big(0, (1+\theta)^{1+1/\theta}\big).
\end{cases}
\end{equation}

\textbf{Orientation:} Theorem \ref{convergence} and its corollary are proved in Section \ref{LimitMeasureSec}, while Theorems  \ref{JointEvaluesGen}, \ref{determinantalThm}, \ref{DTLimitThm} are proved in Sections \ref{DistrSection}, \ref{DeterminantalSection}, \ref{LimitingESDSection} respectively.

\section{The limiting empirical spectral distribution} \label{LimitMeasureSec}

In this and the next section we will use some notions from graph theory. For us, a graph is an ordered triple $(V, E, \phi)$, where $V, E$ are two sets (called the sets of vertices and edges respectively), and $\phi$ is a map from $E$ to $\{\{x, y\}: x, y\in V\}$. The interpretation is that $\phi(v)$ gives the two vertices that the edge $v$ connects, also called ends of $v$ (see Stanley Vol. 1, pg. 573). Such a graph is not directed, and can have several edges with the same ends (multiple edges) and edges with both ends coinciding (loops).

\subsection{Proof of Theorem \ref{convergence}} \label{ESDProof}
We follow the proof of Theorem 3.7 in \cite{BS}. There all matrix elements are i.i.d., so that everything in that proof transfers to
our case (by just replacing all superdiagonal elements with zero)
except the computation of the moments of the limiting measure. In particular, the first step of that proof shows that we can assume that $X_{1, 1}$ has mean 0 and is bounded. With this additional assumption, we prove that
\begin{equation}\lim_{n\to\infty}\ev\left\{\int x^k d L_n(x)\right\}=\frac{k^k}{(k+1)!}
\end{equation}
for all positive integers $k$. And this will complete the proof. We abbreviate the matrix $X(n)$ to $X$.

\nid We have
\begin{align} \label{traceComputation} \ev\left\{\int x^k d
L_N(x)\right\}&=\frac{1}{n}\ev\tr(S_n^k)=\frac{1}{n^{k+1}}\ev\tr \{(X
X^*)^k\}  \notag \\&=\frac{1}{n^{k+1}}\ev\bigg\{\sum_{1\le i_1, i_2,
\ldots, i_k\le n} (XX^*)_{i_1, i_2} (XX^*)_{i_2, i_3}\cdots
(XX^*)_{i_k, i_1}\bigg\} \notag
\\&=\frac{1}{n^{k+1}}\ev\bigg\{\sum_{\substack{1\le i_1, i_2, \ldots,
i_k\le n\\ 1\le j_1, j_2, \ldots, j_k\le n }} X_{i_1, j_1} X^*_{j_1,
i_2} X_{i_2, j_2} X^*_{j_2, i_3}\cdots X_{i_k, j_k} X^*_{j_k, i_1}
\bigg\}
\notag \\&=\frac{1}{n^{k+1}}\sum_{\substack{1\le i_1, i_2, \ldots,
i_k\le n\\ 1\le j_1, j_2, \ldots, j_k\le n }} \ev(X_{i_1, j_1}
\ol{X}_{i_2, j_1} X_{i_2, j_2} \ol{X}_{i_3, j_2}\cdots X_{i_k, j_k}
\ol{X}_{i_1, j_k}).
\end{align} 
Now for a term with indices $i_1, i_2, \ldots, i_k, j_1, j_2, \ldots,
j_k$, we let $i_{k+1}:=i_1$, $\B{i}:=(i_1, i_2, \ldots, i_k),
\B{j}:=(j_1, j_2, \ldots, j_k) $ 
and consider the graph $G(\B{i},
\B{j})$ with vertex set 
$$V(\B{i},
\B{j})=\{(1, i_1), (1, i_2), \ldots, (1, i_k), (2, j_1), (2, j_2), \ldots,
(2, j_k)\}$$ (its cardinality is not necessarily $2k$ because of
repetitions), set of edges
$$\{(2r-1, \{(1, i_r), (2, j_r)\}), (2r, \{(2, j_r), (1, i_{r+1})\}): r=1, 2, \ldots, k \},$$
which has cardinality $2k$, and the map $\phi$ maps $(x, \{y, z\})$ to $\{y, z\}$.
We call a vertex of the form $(1,i)$ an I-vertex, and a vertex 
of the form $(2,i)$ a J-vertex.

From $G(\B{i}, \B{j})$ we generate a graph $G_1(\B{i}, \B{j})$ by identifying edges with equal ends. Formally, $G_1(\B{i}, \B{j})$ has vertex set $V(\B{i}, \B{j})$, edge set
$$\{\{(1, i_r), (2, j_r)\}, \{(2, j_r), (1, i_{r+1})\}: r=1, 2, \ldots, k \},$$
and the maps $\phi_1$ is the identity map. 

\nid As explained in \cite{BS} (in the proof of relation (3.1.6) there,
pages 49, 50), when we take $n\to\infty$ in
\eqref{traceComputation}, the only indices $(\B{i}, \B{j})$ 
contributing are those for which: 
\begin{enumerate}
 \item The graph $G_1(\B{i}, \B{j})$ is a tree with $k+1$ vertices. 
\item The path $(1, i_1)\to (2, j_1) \to (1, i_2) \to (2, j_2) \to \cdots (1, i_k) \to (2, j_k) \to
(1, i_1)$ traverses each edge of the tree exactly twice, in opposite
directions of course. 
\end{enumerate} 
In fact, the pair $(\B{i}, \B{j})$ defines an \textit{ordered} (also
called \textit{plane}) tree, that is, a tree on which we have
specified an order among the children of each vertex. Among two
vertices with common parent, we declare smaller the one that appears
first in the sequence $(i_1, j_1, i_2, j_2, \ldots, i_k, j_k)$. This
order is not related to the labels of the vertices.

\nid In our case, the fact that $X$ is triangular imposes the
additional restrictions 
\begin{enumerate}[resume]
 \item  $j_1\le i_1, i_2$ and  $j_2\le i_2, i_3,$,..., and  $j_k\le
i_k, i_1$.
\end{enumerate}
That is, each $j$ index is smaller than its
two neighbors. 

Call $\Delta(n, k)$ the set of pairs of indices $(\B{i}, \B{j})$
with elements from $\{1, 2, \ldots, n\}$ that satisfy (i), (ii), (iii)
above, and $\hat \Delta(n, k)$ the subset of it for which 
$\{i_1, i_2, \ldots, i_k\}\cap \{j_1, j_2, \ldots, j_k\}=\emptyset$. A pair $(\B{i}, \B{j})\in \hat \Delta(n, k)$, instead of (iii) above, satisfies the stronger property
 
\begin{enumerate}[resume]
 \item  $j_1<i_1, i_2$ and  $j_2<i_2, i_3,$,..., and  $j_k<
i_k, i_1$.
\end{enumerate}

\begin{figure}[htbp]
 \begin{center}
  \resizebox{13em}{!} {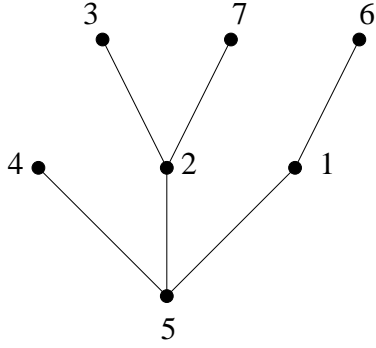} \caption{The tree corresponding to the pair $(\B{i}, \B{j})=((5, 5, 3, 7, 5, 6), (4, 2, 2, 2, 1, 1))$}
 \label{fig2}
 \end{center}
 \end{figure}

An example of a pair $(\B{i}, \B{j})\in \hat \Delta(n, k)$ is $((5, 5, 3, 7, 5, 6), (4, 2, 2, 2, 1, 1))$. Figure \ref{fig2}
shows the tree that this defines. The path
$i_1\to j_1 \to i_2 \to j_2 \to \cdots i_k \to j_k \to
i_1$ travels the tree from left to right.

\begin{lemma} \label{alternatingLemma} For positive integers $n, k$ with $n\ge k+1$, it holds
$$|\hat\Delta(n ,k)|={ n\choose k+1} k^k.$$
\end{lemma}
\begin{proof}

A tree with $r$ vertices labeled $\{1, 2, \ldots, r\}$ is called
\textit{alternating} if for each path $v_1, v_2, \ldots, v_s$ on it we
have 
\begin{align*}
&v_1<v_2>v_3< v_4> \ldots \text{ or }\\ 
&v_1>v_2<v_3> v_4< \ldots
\end{align*}
The set $V(\B{i}, \B{j})$ can take ${ n \choose k+1}$ values. Take
one of them, say $\{1, 2, \ldots, k+1\}$. The indices in $(\B{i},
\B{j})\in \Delta(n ,k)$ for which $V(\B{i}, \B{j})=\{1, 2, \ldots,
k+1\}$ are in a one to one correspondence with the rooted alternating
plane trees with $k+1$ vertices labeled $1, 2, \ldots, k+1$ and such
that the root is larger than its children. Figure \ref{fig2} shows the tree
corresponding to the pair $(\B{i}, \B{j})=((5, 5, 3, 7, 5, 6), (4, 2, 2, 2, 1, 1))$. 
The number of such trees equals $k^k$ (Theorem 3 in \cite{CDR}).  
\end{proof}

\begin{lemma}
$$\lim_{n\to\infty}
\frac{1}{n^{k+1}}|\Delta(n ,k)\sm \hat \Delta(n ,k)|=0.$$
\end{lemma}

\begin{proof}
The elements of $\Delta(n ,k)\sm \hat \Delta(n ,k)$ map injectively to the labeled trees with $k+1$ vertices and at most $k$ labels from $\{1, 2, \ldots, n\}$. The number of such trees is at most
$$\frac{1}{k+1}{2k \choose k}\sum_{j=1}^k (n)_j<2^k k n^k.$$
Here, $(n)_j$ denotes the falling factorial. The lemma follows.
\end{proof}

The expectation in \eqref{traceComputation} corresponding to each
$(\B{i}, \B{j}) \in \Delta(n, k)$ equals 1 due to the assumptions on
the distribution of the $X$'s and property (ii) above. Thus 
$$\lim_{n\to\infty} \ev\left\{\int x^k d
L_N(x)\right\}=\lim_{n\to\infty}
\frac{1}{n^{k+1}}|\Delta(n ,k)|=\lim_{n\to\infty}
\frac{1}{n^{k+1}}|\hat \Delta(n ,k)|=\frac{k^k}{(k+1)!},$$
which concludes the proof of the theorem.

\subsection{The limiting measure. Proof of Corollary \ref{AnalyticProp}} \label{densityPropertiesSec}

(i) By Theorem 2.4.3 in \cite{AGZ} and the analytic properties of
$W$, we get that the measure $\mu$ has support in $[0, e]$ and in
$(0, e]$  has density
$$f(x)=\pi^{-1}\Im\left(\frac{-x^{-1}}{W(-x^{-1})}\right).$$
To exclude the possibility of mass at 0, we show that the
integral of $f$ in $(0, 1]$ is 1. 
First, for $\gd>0$, we compute 
\begin{align}\label{Integral1}
\int_\gd^e \frac{-x^{-1}}{W(-x^{-1})} \,
dx&=\int_{-\gd^{-1}}^{-e^{-1}}
\frac{1}{y W(y)} \, dy=\int_{-\gd^{-1}}^{-e^{-1}}
\frac{1+W(y)}{W(y)^2} W'(y) \, dy\\
&=\int_{-\gd^{-1}}^{-e^{-1}}
\left(-\frac{1}{W(y)}+\log W(y)\right)'\,
dy=1+\log(-1)+\frac{1}{W(-\gd^{-1})}-\log W(-\gd^{-1}). \notag
\end{align}
The second equality is true because the equation defining $W$
easily gives that $W(y)=y \{1+W(y)\} W'(y)$.
Now for $x\to-\infty$, we have $\Re(W(x))\to\infty$
and $|\Im(W(x))|\le\pi$. Thus $\Im (\log W(x))\to 0$, and
the
imaginary part of the integral in \eqref{Integral1} converges to $\pi$
as $\gd\to0^+$. So that the integral of $f$ in $(0, e]$ equals 1.

\smallskip

(ii) For $z\in \CC, |z|>e$, we have
\begin{equation}\label{StieltjesComp} S(z)=-\sum_{k=0}^\infty \frac{1}{z^{k+1}} m_k=-\sum_{k=1}^\infty (k-1)^{(k-1)} \frac{(1/z)^k}{k!}=-L(1/z),
\end{equation}
where $L$ is the exponential generating function of the sequence
$\{(k-1)^{(k-1)}: k\ge 1\}$. It is shown in Lemma 1 of \cite{CDR} that $L$
satisfies
$$L(u)-1=-e^{\frac{u}{L(u)-1}}$$
for $u\in\CC, |u|<e^{-1}$ so that $G(u):=u/(L(u)-1)$
satisfies $G(u)e^{G(u)}=-u$. And since $G(u)\in\R$ for $u\in (-e^{-1}, e^{-1})$, we have $G(u)=W(-u)$ for all $u\in\CC, |u|<e^{-1}$.
Combining $L(u)=1+(u/G(u))$ with \eqref{StieltjesComp}, we get the claim for all $z\in \CC, |z|>e$. The rest follows form the fact that both
$S$ and $W$ are analytic in $\{z\in \CC: \Im(z)>0\}$.

(iii) By definition, $R(z)=K(z)-z^{-1}$ where $K$ satisfies $S(K(z))=-z$. Omitting the analytic details, we note that this is written as $-1+e^{W(-K(z)^{-1})}=-z$ so that 
$$-K(z)^{-1}=W^{-1}(\log(1-z))=\log(1-z)
e^{\log(1-z)}=(1-z)\log(1-z),$$
proving the claim. \hfill \qedsymbol 

\begin{proof}[Proof of Remark \ref{edgesRem}]:
For $x>0$, let $a(x):=\Re(W(-x^{-1})), b(x):=\Im(W(-x^{-1}))$. Then 
\begin{equation} \label{densityFormula}
f_T(x)=\frac{1}{x \pi} \frac{b(x)}{\{a(x)^2+b(x)^2\}}.
\end{equation}

\nid 1. \textsc{Behavior near} 0. 
By \cite{CGHJK}, Section 4, $\lim_{x\to0^+} b(x)=\pi$, while it is an easy exercise to see that $\lim_{x\to0^+} a(x)=\infty$ and
$$\lim_{x\to0^+}\{a(x)-W(x^{-1})\}=0.$$ 
Thus the first $\sim$ has been proved. The second is elementary. 

\nid 2. \textsc{Behavior near} $e$. Since $W(-e^{-1})=-1$, we have $a(e)=1, b(e)=0$, and the denominator in \eqref{densityFormula} is $\sim 1$ as $x\to e^-$. If we let $p(z)=\sqrt{2(1-ez^{-1})}$, relation 4.22 in \cite{CGHJK} implies that $b(x)\sim \Im p(x)$ as $x\to e$. But for $x\in(0, e)$, $p(x)=i \sqrt{2/x}\sqrt{e-x}$. This proves our claim. \end{proof}

\section{The largest eigenvalue. Proof of Theorem \ref{MaxEVConv}}

Theorem \ref{AnalyticProp} and Corollary \ref{convergence} give that $\varliminf \gl_1^n\ge e$. The aim of this section will be to show that $\varlimsup \gl_1^n\le e$. Our proof parallels the one of the Bai-Yin theorem as given in Section 2.3 of \cite{Tao}. The idea is to control a high enough moment of the maximum eigenvalue, and this is accomplished in the next proposition. 

\begin{proposition}
Fix $C_1, C_2>0, \, \gep\in (0, 1/2)$. There exists positive integer $n_0$ with the following property. For $n\ge n_0$, if the support of $|X_{1, 1}|$ is contained in $[-C_1 n^{1/2-\gep}, C_1 n^{1/2-\gep}]$ and $k$ is an integer with $1\le k\le C_2 \log^2 n$, then
\begin{equation}
\ev \tr\{ (X(n) X(n)^*)^k\} \le 2 e^k n^{k+1}. 
\end{equation}

\end{proposition}

\begin{proof}

Let $d_n:=C_2\log^2n$. We pick $n_0$ so that for all $n\ge n_0$ it holds 
\begin{align}\label{n0Choice1}  (2d_n)^6&<n, \\
  (1+2 C_1^2) (2d_n)^{48}&<n^{2\gep}. \label{n0Choice2}
\end{align}
Take $n, k$ as in the statement of the proposition. As in \eqref{traceComputation}, we write  
\begin{align}\ev \tr\{ (X(n) X(n)^*)^k\}&=\sum_{\substack{1\le i_1, i_2, \ldots,
i_k\le n\\ 1\le j_1, j_2, \ldots, j_k\le n }} \ev(X_{i_1, j_1}
\ol{X}_{i_2, j_1} X_{i_2, j_2} \ol{X}_{i_3, j_2}\cdots X_{i_k, j_k}
\ol{X}_{i_1, j_k}) \notag \\&=\sum_{\B{i}, \B{j}} \ev(X_{G(\B{i, j})}) \notag \\
&\le\sum_{\B{i}, \B{j}} \ev(|X_{G(\B{i, j})}|),
\label{GraphsSum}
\end{align}
where to the pair $(\B{i, j})=((i_1, i_2, \ldots, i_k), (j_1, j_2, \ldots, j_k))$ of indices, we correspond the graph $G(\B{i},
\B{j})$ as in Subsection \ref{ESDProof} and the term $X_{G(\B{i, j})}:=X_{i_1, j_1}
\ol{X}_{i_2, j_1} X_{i_2, j_2} \ol{X}_{i_3, j_2}\cdots X_{i_k, j_k}
\ol{X}_{i_1, j_k}$. 

In the sum \eqref{GraphsSum}, we isolate the pairs $(\B{i, j})$ that satisfy (i), (ii), (iv) in the proof of Theorem \ref{convergence}. We call these pairs good, and the rest, bad. The contribution of the good pairs to the sum is 
$${n \choose k+1} k^k=n^{k+1} \frac{k^k}{(k+1)!}<n^{k+1} e^k.$$
The inequality follows by the series expansion of $e^k$.

Now we need to bound the contribution of the bad pairs to \eqref{GraphsSum}. Take such a pair $(\B{i, j})$. The path 
\begin{equation} \label{cycle}
(1, i_1)\to (2, j_1) \to (1, i_2) \to (2, j_2) \to \cdots (1, i_k) \to (2, j_k) \to
(1, i_1)
\end{equation}
 is a cycle that traverses the graph $G_1(\B{i},
\B{j})$. List the edges $e_1, e_2, \ldots, e_s$ of $G_1(\B{i},
\B{j})$  in order of appearance in the cycle, and call $a_1, a_2, \ldots, a_s$ their multiplicities in the cycle. That is, $a_q$ is the number or times the (undirected) edge $e_q$ appears in the cycle. If any of these multiplicities is 1, we have $\ev(X_{G(\B{i, j})})=0$. We assume therefore that all are at least 2. Using the information about the mean, variance, and support of $|X_{1, 1}|$, we get that for $a\ge 2$ integer it holds $\ev(|X_{1, 1}|^a)\le (C_1 n^{1/2-\gep})^{a-2}.$ Thus 
\begin{equation}\ev(|X_{G(\B{i, j})}|) \le \prod_{i=1}^s \ev |X_{e_i}|^{a_i} \le (C_1 n^{1/2-\gep})^{a_1+\cdots+a_s-2s}=  (C_1 n^{1/2-\gep})^{2k-2s}.
\end{equation}

Cycles that are generated by bad pairs we call them bad cycles. 
For integers $s\ge1$ and $a_1, \ldots, a_s\ge2$, let $N_{a_1, a_2, \ldots, a_s}$ be the number of bad cycles whose edges have multiplicities 
$a_1, a_2, \ldots, a_s$. The contribution of the bad pairs to \eqref{GraphsSum} is at most
\begin{equation} \label{FirstGraphBound}
\sum_{s=1}^{k} (C_1 n^{1/2-\gep})^{2k-2s} \sum_{a_1, a_2, \ldots, a_s} N_{a_1, a_2, \ldots, a_s}.
\end{equation}
Using Lemma \ref{CycleCountLemma}, we bound the last sum by 
\begin{equation} \label{SecondGraphBound} e^k (2k)^{10} \sum_{s=1}^{k} (C_1 n^{1/2-\gep})^{2k-2s} (2k)^{ 36(k-s)} n^{\min\{s+1, k\}}  \sum_{a_1, a_2, \ldots, a_s} 1.
\end{equation}
The inside sum is over all $s$-tuples of integers greater than or equal to 2 with sum $2k$. By subtracting 2 from each $a_i$, we get an $s$-tuple of non-negative integers with sum $2k-2k$.  The number of such $s$-tuples is ${2k-s-1 \choose 2k-2s}$ (combinations with repetition) which is at most $(2k)^{2(k-s)}$. Separating the $s=k$ term, and letting $w=k-s$, we get for \eqref{SecondGraphBound} the bound 
$$e^k (2k)^{10}\left\{n^k+n^{k+1}\sum_{w=1}^{k-1} \left(\frac{C_1^2 (2k)^{38}}{n^{2\gep}}\right)^w \right\}.$$
By the choice of $n_0$, we have $C_1^2 (2k)^{38}/n^{2\gep}<1/2$, and the sum in the last expression is bounded by $2C_1^2 (2k)^{38}/n^{2\gep}$. Thus, the contribution of the bad pairs to the sum \eqref{GraphsSum} is at most 
$$e^k (2k)^{10}\left(n^k + n^{k+1-2\gep} 2 C_1^2 (2k)^{38} \right)\le e^k n^{k+1-2\gep} (1+2C_1^2) (2k)^{48}<e^k n^{k+1}.$$
In the last equality, we used again the choice of $n_0$.
This finishes the proof of the proposition.
\end{proof}

Now Theorem \ref{MaxEVConv} follows by adapting the arguments of Theorems 2.3.23, 2.3.24 (the Bai-Yin Theorem) in \cite{Tao}. In the rest of the section, we prove the crucial estimate we invoked in the proof above.

\begin{lemma} \label{CycleCountLemma} $N_{a_1, a_2, \ldots, a_s}\le e^k (2k)^{ 36(k-s)+10} n^{\min\{s+1, k\}}$.
\end{lemma}

\begin{proof} Take a cycle as in \eqref{cycle}, and label the vertices as $v_1\to v_2 \to \cdots \to v_{2k}\to v_{2k+1}=v_1$. Each step in the cycle we call a leg. More formally, legs are the elements of the set $\{(a, (v_a, v_{a+1})): a=1, 2,\ldots, 2k\}$.  For $1\le a<b$, we say that the leg $(a, (v_a, v_{a+1}))$ is single up to $b$ if $\{v_a, v_{a+1}\}\ne \{v_c, v_{c+1}\}$ for every $c\in\{1, 2, \ldots, b-1\}$. We classify the $2k$ legs of the cycle into 4 sets $T_1, T_2, T_3, T_4$. The leg $(a, (v_a, v_{a+1}))$  belongs to 

\smallskip

$T_1$: if $v_{a+1}\notin \{v_1, \ldots, v_a\}$. 

\smallskip

$T_3$: if there is $T_1$ leg $(b, (v_b, v_{b+1}))$ with $b<a$ so that $a=\min\{c>b: \{v_c, v_{c+1}\}=\{v_b, v_{b+1}\}\}$.

\smallskip

$T_4$: if it is not $T_1$ or $T_3$.

\smallskip

$T_2$: if it is $T_4$ and there is no $b<a$ with $\{v_a, v_{a+1}\}=\{v_b, v_{b+1}\}$.

\nid Moreover, a $T_3$ leg $(a, (v_a, v_{a+1}))$ is called irregular if there is exactly one $T_1$ leg $(b, (v_b, v_{b+1}))$ with $b<a$ satisfying $\{v_c, v_{c+1}\}\ne\{v_b, v_{b+1}\}$ for all $c\in\{1, 2, \ldots, a-1\}\sm\{b\}$ and $v_a\in \{v_b, v_{b+1}\}$. Otherwise the leg is called regular. 

\nid We already know that the number of edges of  $G_1(\B{i},
\B{j})$ is $s$. Let also

\smallskip

$t$: the number of vertices of $G_1(\B{i},
\B{j})$. 

\smallskip

$\ell$: the number of edges that have multiplicity at least 3. 

\smallskip

$m$: the number of $T_2$ legs.

\smallskip

$r$: the number of regular $T_3$ legs.

\smallskip

\nid The number of edges with multiplicity 2 is $s-\ell$. Thus, 
\begin{equation}
\sum_{i: a_i>3} a_i=2k-2(s-\ell).
\end{equation}
On the other hand, the same sum is at least $3\ell$. Thus $\ell\le 2(k-s)$. Lemma 2.3 in \cite{BY} says that $r\le 2|T_4|$, while for $|T_4|$ we have the bound
\begin{equation}|T_4| \le m+\sum_{i:a_i\ge 3} a_i\le m+6(k-s).
\end{equation}
Now back to the task of bounding $N_{a_1, \ldots, a_s}$.
Given a cycle, we give each vertex an index in $\{1, 2, \ldots, t\}$ which records the order of the first appearance of the vertex in the cycle. 

\nid Then, we record
\begin{itemize}
\item the locations of regular $T_3$ legs in the cycle and the index of the vertex each has as a final vertex. 

\item the locations of $T_4$ legs in the cycle and the index of the vertex each has as a final vertex. 

\item the index of each J vertex, say $(2, j)$, for which $j\in\{i_1, i_2, \ldots, i_k\}$.
\end{itemize}
Finally, a cycle defines a rooted, ordered, alternating tree with $t$ vertices and labels $\{1, 2, \ldots, t\}$. We get this tree as follows.

\begin{itemize} 
\item[1.] Graph: In $G_1(\B{i, j})$, we erase edges that were traveled by $T_2$ legs in $G(\B{i, j})$. We thus get a simple graph $\hat G(\B{i, j})$ (i.e., with no multiple edges) which is in fact  a spanning tree of  $G(\B{i, j})$. Indeed, it has the same set of vertices as $G(\B{i, j})$ and is connected because the edges we erased connect vertices that were already connected by a different route in $G(\B{i, j})$. And it is a tree because if there were a simple cycle in it, we would be able to find in it an edge traveled by a $T_2$ leg of $G(\B{i, j})$, which is false. Root of the tree is the vertex $(1, i_1)$, and the children of each vertex are ordered according to their index.

\item[2.] Labels: Initially, to each vertex $(a, b)$ of this tree (recall that $a\in\{1, 2\}$) we assign the label $b$. 
Next, we assign new labels so that the labels of any two vertices have the same order as before, but the labels used form an initial segment of the positive integers. Assume that they are $\{1, 2, \ldots, t-u\}$ for some integer $u\ge0$. It is $u>0$ exactly when $u$ J-vertices have label that agrees with the label of one I-vertex. We do now a final relabeling.
If a J vertex $v:=(2, j)$ has $j\in\{i_1, i_2, \ldots, i_k\}$, we increase by one the label of every vertex (I or J) which has at the moment label $\ge j$ except $v$. We do this procedure sequentially by checking equality for the label 1 and continuing upward. In the end, no two vertices will have the same label, and the set of labels will be $\{1, 2, \ldots, t\}$.     
\end{itemize}

\nid Step 1 gives an ordered rooted tree of $t$ vertices, and Step 2 together with the property 
$$j_1\le i_1, i_2 \text{ and  } j_2\le i_2, i_3,,..., \text{ and } j_k\le
i_k, i_1$$
that the indices have, gives us that the labeling is alternating (see definition in the proof of Lemma \ref{alternatingLemma}) with the root larger than its children.

Having these elements, we can reconstruct the cycle up to the names of the $t-u$ different labels (the locations of the labels are determined by the labeled tree). Because the locations of all legs in the cycle are either known or can be inferred. The same holds for the ending point of each leg. Since all legs start at the ending point of some other leg, the location of the legs in the cycle determines the starting points too.

\nid Thus, the number of bad cycles with given $t, u,  r, |T_4|$ is at most
\begin{align}
 {n \choose t-u} (t-1)^{t-1} (2k)^r   t^r (2k)^{|T_4|} t^{|T_4|}  t^u & \le n^{t-u}  \frac{(t-1)!}{(t-u)!}\,  e^{t-1} (2k)^{2(r+|T_4|)}t^u
\\ & \le  e^{t-1} n^t \left(\frac{t^2}{n}\right)^u(2k)^{2(r+|T_4|)}. \label{MapCount}
\end{align}
We used the rough bound $t\le 2k$. By the choice of $n_0$, we have $t^2<n$. Moreover, since $t\le s+1-m\le k+1$ and 
\begin{align}
r+|T_4| &\le 3|T_4| \le 3m+18(k-s),
\end{align}
the bound in \eqref{MapCount} is less than
\begin{equation}\label{CyclesBound} e^k n^{s+1-m} (2k)^{6m+36(k-s)}=e^k (2k)^{36(k-s)} n^{s+1} \left(\frac{(2k)^6}{n}\right)^m\le e^k (2k)^{36(k-s)+6} n^{\min\{s+1, k\}}.
\end{equation}
The last inequality is true because $s\le k$ always, $(2k)^6/n<1$ by the choice on $n_0$, and because $s=k$ implies $m\ge1$ as the cycle is bad. Summing the bound \eqref{CyclesBound} over all possible values of $t, r, u, |T_4|$, which are at most $2k$ for each, we get the claim of the lemma. 
\end{proof}

\section{Distribution of singular values for $X^{\theta, b}(n)$} \label{DistrSection}

Define the following sets of matrices

\smallskip

 $\C{T}_n$: lower triangular $n\times n$ matrices with elements in
$\CC$ and diagonal elements in $\CC\sm\{0\}$.

 $\C{T}_n^+$: elements of $\C{T}_n$ with diagonal
elements in $(0, \infty)$,

$\C{V}_n$: diagonal $n\times n$ matrices
with diagonal elements complex of modulus 1. 

$\C{M}_n^+$: positive definite $n\times n$ matrices with elements in
$\CC$.

\smallskip

\nid We identify the spaces $\C{T}_n, \C{T}_n^+, \C{V}_n$ with
$\R^{n(n-1)}\times (\R\sm \{0\})^{2n} , \R^{n(n-1)}\times (0,
\infty)^n, [0, 2\pi)^n$ respectively, and view $\C{M}_n^+$ as a subset of $n\times n$ Hermitian matrices, which we identify with $\R^{n^2}$. Densities of random variables with values in these spaces are meant with respect to the corresponding Lebesgue measure. 
 
Consider the maps $g:\C{T}_n^+\times \C{V}_n \to \C{T}_n,  h: \C{T}_n^+\to \C{M}_n^+$ with $g(T,
V)=TV$,  $h(Y):=Y Y^*$. They are both one to one and onto. Call $g^{-1}:=(\gga_1, \gga_2)$, and $X:=X^{\theta, b}(n)$. Then $XX^*=h(\gga_1(X))$ provided that $X\in \C{T}_n$, which holds with probability 1. We will use this relation in order to find the joint law of the elements of $XX^*$, and then, the law of its eigenvalues will follow from a well known formula.

\begin{lemma} \label{gJacob}
The Jacobian of the map $g$ has absolute value
$$\prod_{j=1}^n t_{j, j}.$$
\end{lemma}
\begin{proof} 
Let $X=TV$ and call $x_{i, j}$ its $(i, j)$ element.
The Jacobian matrix of $g$ has $n$ blocks, one for each column. The block corresponding to column $j$ is 
$$\frac{\partial(x_{j, j}^R, x_{j, j}^I, x_{j+1, j}^R, x_{j+1, j}^I, \ldots, x_{n, j}^R, x_{n, j}^I)}{\partial(\theta_j, t_{j, j}, t_{j+1, j}^R, t_{j+1, j}^I, \ldots, t_{n, j}^R, t_{n, j}^I )},$$
and its determinant equals $-t_{j, j}$. 
\end{proof}

\begin{lemma} \label{hJacob} The map $h$ has Jacobian
$$2^n \prod_{i=1}^n t_{i, i}^{2(n-i)+1}.$$
\end{lemma}
\begin{proof} This is Proposition 3.2.6 of \cite{For}. 
\end{proof}

In the following, we use the notation set in Subsection \ref{LaguerreSec}. Let $C(\theta, b):=\big(\prod_{k=1}^n \Gamma(c_k+1)\big)^{-1}$. The density of $X^{\theta, b}(n)$ is
\begin{align} \label{Xdistribution} f_{X^{\theta, b}(n)}(x)&=\frac{1}{\pi^{n(n+1)/2}}\, C(\theta, b)\,  e^{-\sum_{1\le j \le i\le n}
|x_{i, j}|^2} \prod_{k=1}^n |x_{k, k}|^{2c_k}\\&=
\frac{1}{\pi^{n(n+1)/2}}\, C(\theta, b)\,  e^{-\tr(x x^*)}\prod_{k=1}^n |x_{k, k}|^{2c_k}
\end{align}
for all $x\in \CC^{n(n+1)/2}$.

For an $n\times n$ matrix $a=(a_{i, j})_{1\le i, j\le n}$ and
$k\in\{1, 2, \ldots, n\}$, we denote by $a_k$ its main $k\times k$
minor, that is, the matrix $(a_{i, j})_{1\le i, j\le k}$.

\begin{proposition} \label{XX*Density} Let $X:=X^{\theta, b}(n)$. The matrix $A:=X X^*$ has density 
\begin{equation} \label{Adensity}
f_A(a)=\frac{1}{\pi^{n(n-1)/2}}\frac{1}{\prod_{k=1}^n \Gamma(c_k)}\, e^{-\tr(a)}\{\det(a)\}^{c_n-1} \{\det(a_1) \det(a_2)\cdots
\det(a_{n-1})\}^{-\theta-1}
\end{equation}
for all $a\in \C{M}_n^+$, and $f_A(a)=0$ for every Hermitian matrix not an element of $\C{M}_n^+$.
\end{proposition}

\begin{proof} Since $XX^*=h(\gga_1(X))$, our first step is to find the distribution of $T:=\gga_1(X)$.

\smallskip

\textsc{Claim}: $T$ has density 
$$f_T(t)=(2\pi)^n f_X(t) \prod_{j=1}^n t_{j, j}.$$

\nid \textsc{Proof of the claim}: For any positive measurable function defined on $\C{T}_n$, we have  
\begin{align*}\ev\{s(T)\}&=\ev\{s(\gga_1(X))\}=\int s(r_1(x)) f_X(x)\,  dx=\int_{\C{T}_n^+} \int_{[0, 2\pi)^n} s(\gga_1(g(t, \theta)) f_X(g(t, \theta))|Jg(t, \theta)| \, d\theta \, dt \\
&=(2\pi)^n \int_{\C{T}_n^+} s(t) f_X(t) \prod_{j=1}^n t_{j, j}\, dt.
\end{align*}
In the last equality we used Lemma \ref{gJacob}, and the fact that $f_X(T V)=f_X(T)$ for all $V$ of the form $\text{diag}(e^{i\gt_1}, e^{i\gt_2}, \ldots, e^{i\gt_n})$. Thus, the claim is proved.

\smallskip

\nid Now, for given $a\in \C{M}_n^+$, let $t:=h^{-1}(a)$. Then 
\begin{align*}f_A(a)&=f_T(h^{-1}(a))
|Jh^{-1}(a)|=(2\pi)^n f_X(h^{-1}(a)) 
\bigg(\prod_{j=1}^n t_{j, j}\bigg) \frac{1}{|Jh(h^{-1}(a))|}\\&=(2\pi)^n 
\frac{1}{\pi^{n(n+1)/2}}\, C(\theta, b) \, e^{-\tr(a)} \prod_{j=1}^n |t_{j, j}|^{2(c_j-1)} \frac{1}{2^n \prod_{j=1}^n t_{j, j}^{2(n-j)+1}} \prod_{j=1}^n
t_{j, j} \\
&=\frac{1}{\pi^{n(n-1)/2}} \, C(\theta, b)\, e^{-
\tr(a)}\bigg(\prod_{i=1}^n t_{j, j}^{2(n-j-c_j+1)}\bigg)^{-1}
\\&=\frac{1}{\pi^{n(n-1)/2}} \, C(\theta, b) \, e^{-\tr(a)} \bigg(\prod_{j=1}^n t_{j, j}^2\bigg)^{c_n-1}\bigg(\prod_{j=1}^n t_{j, j}^{2(n-j)}\bigg)^{-(1+\theta)}.
\end{align*}
In the third equality we used Lemma \ref{hJacob}, and in the last equality the fact that $-c_j=\theta(n-j)-c_n$ for all  $j\in\{1, 2, \ldots, n\}$. Finally, we express the products involving the variables $t_{j, j}$ in terms of the variable $a$. Since $T$
is lower triangular, we have $a_i=T_i T_i^*$. Thus 
$$\det(a_i)=|\det(T_i)|^2=(t_{1, 1} t_{2, 2}\ldots t_{i, i})^2.$$
Multiplying these equalities for all $1\le i\le n-1$, we get
$$\det(a_1) \det(a_2)\cdots \det(a_{n-1})=\prod_{i=1}^n t_{i,
i}^{2(n-i)}.$$
This finishes the proof of the proposition.
\end{proof}

\begin{proof}[Proof of Theorem \ref{JointEvaluesGen}]
From relations (4.1.17), (4.1.18) in \cite{AGZ}, and the fact that $X^{\theta, b}(n)X^{\theta, b}(n)^*$ is positive definite,  we have
that the vector of the eigenvalues in decreasing order has density
$$f_{\gL_n}(\gl)=C_n \prod_{i<j}(\gl_i-\gl_j)^2 \int_{U(n)}f_A(H
D_{\gl} H^*) (dH)\, \one_{\gl_1>\gl_2> \ldots
>\gl_n>0}$$
where $\gl:=(\gl_1, \gl_2, \ldots, \gl_n), D_\gl$ is the
diagonal matrix with diagonal $\gl$, $(dH)$ is the normalized Haar
measure on $U(n)$, and the constant $C_n$ is
$$C_n:=\frac{\pi^{n(n-1)/2}}{\prod_{j=1}^{n-1} j!}.$$
Thus, writing $a:=H D_{\gl} H^*$ and taking into account Proposition
\ref{XX*Density}, we get
\begin{equation} \label{evDensityStep}
f_{\gL_n}(\gl)=\frac{\, C(\theta, b)}{\prod_{j=1}^{n-1} j!}\, \Big\{\prod_{i<j}(\gl_i-\gl_j)^2 \Big\}\, 
e^{-\sum_{j=1}^n\gl_j} \bigg(\prod_{j=1}^n \gl_j\bigg)^{c_n-1} K(\gl)\, \one_{\gl_1>\gl_2> \ldots
>\gl_n>0}
\end{equation}
with
\begin{equation} \label{UnitIntegral}
K(\gl):=\int_{U(n)} \{\det(a_1) \det(a_2)\cdots
\det(a_{n-1})\}^{-\theta-1}\, (dH).
\end{equation}
The computation of the last integral is given in Lemma \ref{BariLemma}. Combining that computation with \eqref{evDensityStep}, we finish the proof.
\end{proof}

\begin{lemma} \label{BariLemma} For $\theta\ge0$, the integral in \eqref{UnitIntegral} equals 
\begin{align}
K(\gl)=\frac{\prod_{1\le i<j\le n}\int_{\gl_j}^{\gl_i} x^{-\theta-1}\, dx}{{\prod_{1\le i<j\le n}(\gl_i-\gl_j)}} =\left(\prod_{i=1}^n \gl_i \right)^{-\theta (n-1)} \frac{\prod_{1\le i<j \le n} \int_{\gl_j}^{\gl_i} x^{\theta-1}\, dx}{\prod_{1\le i<j \le n}(\gl_i-\gl_j)}.
\end{align}
\end{lemma}

\begin{proof}
To simplify the exposition, we introduce a binary relation which we denote by $\succ$.
$x\succ y$ means that there is $k\in \Np$ so that $x=(x_1, x_2,
\ldots, x_{k+1})\in \R^{k+1}, y=(y_1, y_2,
\ldots, y_k)\in \R^k$, and  
$$x_1\ge y_1\ge x_2 \ge y_2\cdots \ge x_n\ge y_n \ge x_{k+1}.$$
For $x=(x_1, x_2, \ldots, x_k)\in \R^k$,
we let $Y(x)$ be the set of
all elements $(y^{(k-1)}, y^{(k-2)}, \ldots, y^{(1)})$ of
$\R^{k-1}\times
\R^{k-2}\times \cdots \times \R^2 \times \R$ that satisfy
$$\gl \succ y^{(k-1)} \succ y^{(k-2)} \succ \cdots \succ y^{(2)} \succ
y^{(1)}.$$  
One can easily verify that
\begin{equation} \label{GCVolume}
\operatorname{Vol}(Y(x))=\prod_{1\le i<j \le k} \frac{x_i-x_j}{j-i}.
\end{equation}

We can now start the proof of the lemma.
For each $i=1, 2, \ldots, n-1$, call $x^{(i)}=(x_1^{(i)}, x_2^{(i)},
\cdots, x_i^{(i)})$ the vector of the eigenvalues of the
symmetric matrix $a_i$ with $x_1^{(i)}\ge x_2^{(i)} \ge
\cdots \ge x_i^{(i)}$.  
Then the integrand in \eqref{UnitIntegral} is simply
$$\prod_{i=1}^{n-1} \prod_{j=1}^i (x_j^{(i)})^{-\theta-1}.$$ 
Under $(dH)$, the law of $a$ is the one of an
$n\times n$ GUE matrix conditioned to have eigenvalues $\gl_1, \gl_2,
\ldots, \gl_n$, and according to Proposition 4.7 in \cite{Bar}, the law
of $(x^{(n-1)}, x^{(n-2)}, \ldots, x^{(1)})$ is the uniform on
$Y(\gl)$ with respect to Lebesgue measure. Thus the integral equals
\begin{equation}  \label{multipleIntegral}
\frac{1}{\operatorname{Vol}(Y(\gl))}
\underbrace{\int_{\gl_n}^{\gl_{n-1}} \int_{\gl_{n-1}}^{\gl_{n-2}}
\cdots \int_{\gl_2}^{\gl_1}}_\text{$n-1$ integrals} \,  \underbrace{
\int_{x_{n-1}^{(n-1)}}^{x_{n-2}^{(n-1)}} \cdots
\int_{x_2^{(n-1)}}^{x_1^{(n-1)}}}_\text{$n-2$ integrals} \cdots
\underbrace{\int_{x_2^{(2)}}^{x_1^{(2)}}}_\text{1 integral}
\prod_{1\le j\le i \le n-1} (x_j^{(i)})^{-\theta-1} \prod_{1\le j\le i \le n-1} dx_j^{(i)}.
\end{equation}
Assume now that $\theta>0$. Let $\tilde \gl(\theta):= (\gl_n^{-\theta}, \gl_{n-1}^{-\theta}, \ldots, \gl_1^{-\theta})$,
and in the integral make the change of variables $y^{(i)}_j=(x^{(i)}_j)^{-\theta}$ 
for all $1\le j\le
i\le n-1$. Then the previous expression becomes 
\begin{align} \notag
\frac{1}{\operatorname{Vol}(Y(\gl))}
(-\theta)^{-n(n-1)/2} &
\underbrace{\int_{\gl_n^{-\theta}}^{\gl_{n-1}^{-\theta}}
 \int_{\gl_{n-1}^{-\theta}}^{\gl_{n-2}^{-\theta}}
\cdots \int_{\gl_2^{-\theta}}^{\gl_1^{-\theta}}}_\text{$n-1$ integrals}  \, 
\underbrace{
\int_{y_{n-1}^{(n-1)}}^{y_{n-2}^{(n-1)}} \cdots
\int_{y_2^{(n-1)}}^{y_1^{(n-1)}}}_\text{$n-2$ integrals} \cdots
\underbrace{\int_{y_2^{(2)}}^{y_1^{(2)}}}_\text{1 integral}
 \prod_{i=1}^{n-1}
\prod_{j=1}^i dy_j^{(i)}
\\=\theta^{-n(n-1)/2} \frac{\operatorname{Vol}(Y( \tilde \gl(\theta)
))}{\operatorname{Vol}(Y(\gl))}&=\theta^{-n(n-1)/2} \frac{\prod_{1\le i<j \le n} (
\gl_j^{-\theta}-\gl_i^{-\theta})}{\prod_{1\le i<j \le n}(\gl_i-\gl_j)}\\&=\left(\prod_{i=1}^n \gl_i \right)^{-\theta (n-1)} \frac{\theta^{-n(n-1)/2}\prod_{1\le i<j \le n} (
\gl_i^{\theta}-\gl_j^{\theta})}{\prod_{1\le i<j \le n}(\gl_i-\gl_j)}.
\end{align}
And the lemma is proved in this case. In the case  $\theta=0$, in the integral of \eqref{multipleIntegral}, we let $y^{(i)}_j=\log x^{(i)}_j$ for all $1\le j \le i \le n-1$ and proceed as above. Alternatively, we can take $\theta\to0$ in the last expression. 
\end{proof}

\section{Determinantal process. Proof of Theorem \ref{determinantalThm}} \label{DeterminantalSection}

\begin{proof}
(i). For each positive integer $n$ and $y_1, y_2, \ldots, y_n\in \D{R}$, we let
$$\Delta(y_1, y_2, \ldots, y_n):=\det(y_k^{j-1})_{1\le j, k\le n}=\prod_{1\le j<k\le n}(y_k-y_j).$$
Equation (3.3) of \cite{DG} gives that the determinant of $G^{(n)}$ is 
\begin{align}
\det\left(\int_0^\infty x^j (\log x)^k e^{-x}\, dx \right)_{0\le j, k<n}&=\frac{1}{n!} \int_0^\infty \cdots \int_0^\infty \det(x_k^{j-1})_{1\le j, k\le n} \det((\log x_k)^{j-1})_{1\le j, k\le n} \prod_{k=1}^n dx_k\\
&=\frac{1}{n!} \int_0^\infty \cdots \int_0^\infty \Delta_n(x_1, \ldots, x_n) \Delta_n(\log x_1, \ldots, \log x_n) \prod_{k=1}^n dx_k
\end{align}
The integrand is positive, thus the determinant is not zero.

\medskip

(ii) Follows from part (i) and Proposition 5.8.1 of \cite{For}.
\end{proof}

\nid In the rest of the section, we discuss the structure of $G^{(n)}$ and compute explicitly the value of its determinant.
\begin{lemma}
The matrix $G:=(g_{i, j})_{i, j\in\N}$ has an $LU$ factorization $G=LU$ with 
\begin{align}
L_{i, j}&=|s(i+1, j+1)| & \text{ for } i\ge j \ge 0,\\
U_{i, j}&=(j)_i g_{0, j-i} & \text{ for } 0\le i \le j.  
\end{align}
Here $s$ denotes the Stirling number of the first kind.
\end{lemma}
We follow the convention $(0)_i=1$ for all nonnegative integers $i$.

\begin{proof}
We compute the exponential generating function of the sequence $(g_{j, k})_{j, k\in \N}$.
\begin{align*}
\sum_{j, k=0}^\infty \frac{u^j}{j!}\frac{v^k}{k!} g_{j,
k}&=\int_0^\infty e^{-x} e^{ux} e^{v\log x} \, dx=\int_0^\infty e^{-x}
e^{-(1-u)x} x^v \, dx=(1-u)^{-v-1} \Gamma(1+v)\\&=\sum_{j=0}^\infty (v+1)(v+2)\cdots (v+j) \frac{u^j}{j!} \sum_{s=0}^\infty \frac{\gG^{(s)}(1)}{s!} v^s\\&=\sum_{j=0}^\infty \frac{u^j}{j!} \sum_{r=0}^j (-1)^{j-r} s(j+1, r+1) v^r  \sum_{s=0}^\infty \frac{\gG^{(s)}(1)}{s!} v^s\\&=
\sum_{j, k=0}^\infty \frac{u^j}{j!} v^k \sum_{r=0}^{j\wedge k} \frac{\gG^{(k-r)}(1)}{(k-r)!} (-1)^{j-r} s(j+1, r+1).
\end{align*}
Since for $m, n\in \N$ the integer $s(m, n)$, if not zero, has sign $(-1)^{m-n}$ and $\gG^{(n)}(1)=g_{0, n}$, we get  
$$g_{j, k}=\sum_{r=0}^{j\wedge k} |s(j+1, r+1)| (k)_r g_{0, k-r}$$
for all $j, k\in \N$. This proves the factorization $G=L U$. 
\end{proof}

\nid Since $L_{k,k}=1$ and $U_{k, k}=k!$ for all $k\in \N$, we obtain 
$$\det(G^{(n)})=1!2! \cdots (n-1)!.$$ 

\section{Limiting empirical distribution in the $\theta>0$ case} \label{LimitingESDSection}

\textit{Proof of Theorem \ref{DTLimitThm}}: We first prove the following. 

\smallskip

\nid \textsc{Claim}: The sequence $\left(X^{\theta, b}(n)/\sqrt{n}\right)_{n\ge1}$ converges in $*$-moments to a $DT(\nu_\theta,1)$ element.

\smallskip

\nid Recall the form of the diagonal elements of $X^{\theta, b}(n)$ in \eqref{DiagonalLaws}.  Call $\eta_n$ the law of the vector $(X_{k, k}/\sqrt{n})_{1\le k\le n}$, $\tilde \eta_n$ its symmetrization, that is, the law of $(X_{\pi(k), \pi(k)}/\sqrt{n})_{1\le k\le n}$, where $\pi$ is a random permutation of $\{1, 2, \ldots, n\}$ independent of the matrix,
and $\tilde \eta_n^{(p)}$ the projection onto the first $p$ coordinates of $\tilde \eta_n$ ($p\in\{1, 2, \ldots, n\}$). 
According to Theorem 2.13 in \cite{DH}, it is enough to show that $\tilde \eta_n^{(p)}$ converges in $*$-moments to $\displaystyle \times_1^p \nu_\theta$. Since the law $\eta_n$ is a product measure, it is enough to show this only for $p=1$. Clearly the law of $\tilde \eta_n^{(1)}$ is radially symmetric. It suffices therefore to prove that for any function $h:\R\to \R$ of at most polynomial growth, it holds that
\begin{equation}\label{GammaAverage1}
\frac{1}{n}\sum_{k=1}^n \ev h\left(\frac{Y_k^2}{2n}\right)\to\frac{1}{\theta}\int_0^{\theta} h(r)\, dr
\end{equation}
as $n\to\infty$. Take independent random variables $(W_k)_{k\ge1}$ so that $W_1\sim \Gamma(b, 1)$ and $W_k\sim \Gamma(\theta, 1)$ for all $k\ge2$. Then $Y_k^2/2$ has the same law as $S_k:=W_1+W_2+\cdots+W_k$. The left hand side of \eqref{GammaAverage1}  is
\begin{equation} \label{GammaAverage2} \frac{1}{n}\sum_{k=1}^n \ev h\left(\frac{S_k}{k}\frac{k}{n}\right).
\end{equation}
Now, $S_k/k$ converges pointwise to $\theta$, and satisfies a large deviations principle with speed $n$ and good rate function having a unique zero at $\theta$. At the same time, $h$ has at most polynomial growth. It is easy then to show that \eqref{GammaAverage2} converges to $\int_0^1 h(\theta x)\, dx$, proving our claim.

Call $x$ the limit mentioned in the claim. We can assume that $x$ is an element of a von Neumann algebra (see Remark 2.3 in \cite{DH}), and thus there is a unique compactly supported measure in $\R$ that has the same moments as $xx^*$ (Lemma 5.2.19 in \cite{AGZ}). The theorem follows by combining this with the above claim. \qed

\bigskip

\textbf{Acknowledgments:} This research was carried out while I spent the fall semester of 2013 at Leiden University. I thank Frank den Hollander for the invitation, and the entire  probability group for the stimulating atmosphere.     

\bibliography{TBIO}

\end{document}